\documentclass[12pt,oneside,reqno]{amsart}

\usepackage[margin=0.9in]{geometry}
\usepackage{amsmath,amssymb,amsthm,mathtools}
\usepackage{enumitem}
\usepackage{microtype}
\usepackage[colorlinks=true,linkcolor=blue,citecolor=blue,urlcolor=blue]{hyperref}
\usepackage[toc,title,page]{appendix}
\usepackage[skip=5pt,indent=0pt]{parskip}

\numberwithin{equation}{section}

\newtheorem{theorem}{Theorem}[section]
\newtheorem*{theorem*}{Theorem}
\newtheorem{proposition}[theorem]{Proposition}
\newtheorem{lemma}[theorem]{Lemma}

\newtheorem{question}[theorem]{Question}

\theoremstyle{remark}
\newtheorem{remark}[theorem]{Remark}

\newenvironment{manualtheorem}[1]{%
  \renewcommand\thetheorem{#1}
  \theorem
}{\endtheorem}

\newcommand{\R}{\mathbb R}

\newcommand{\eps}{\varepsilon}
\newcommand{\dd}{\,d}
\newcommand{\ip}[2]{\langle #1,#2\rangle}

\newcommand{\vol}{\operatorname{vol}}

\newcommand{\conv}{\operatorname{conv}}

\def\s{\sigma}
\def\class{\mathcal{K}^n_0}

\title[Uncentered Blaschke-Santal\'o inequalities for the Gaussian measure]{Uncentered Blaschke-Santal\'o inequalities for the Gaussian measure}

\author[S. Artstein-Avidan]{S. Artstein-Avidan}
\author[M. Fradelizi]{M. Fradelizi}
\author[K. Wyczesany]{K. Wyczesany}

\date{}

\begin{document}

\begin{abstract}
We study the maximizers of the generalized volume product
\[
        \gamma_\sigma^n(A)\,\gamma_\sigma^n(A^\circ)
\]
among all measurable subsets $A\subset\R^n$, where 
$A^\circ$ denotes the polar set of $A$, and where $\gamma_\sigma^n$ denotes the centered Gaussian probability measure on $\R^n$ with covariance $\sigma^2 I_n$, $\sigma>0$.  It turns out that the maximizers depend on $\sigma$.
We prove that they exist and are convex bodies.
In dimension $n=1$, we find the exact form of the maximizers.  
In dimension $n\ge 2$, we show that they are smooth bodies of revolution whose support function satisfies a certain differential equation. 
Moreover, for $\sigma^2 \le \frac{1}{{n}}$ we show that the Euclidean unit ball is the unique maximizer, while this is no longer the case for $\sigma^2> {\frac{2}{n+1}}$. In dimension $n=2$, we close the gap by showing that the Euclidean unit ball is the unique maximizer for $\sigma^2 \le \frac{2}{3}$.

\end{abstract}

\maketitle

\section{Introduction}

Let $\gamma_\sigma^n$ be the centered Gaussian probability measure on $\R^n$ with variance $\sigma^2>0$, namely
\[
        d\gamma_\sigma^n(x)=\frac{1}{(2\pi\sigma^2)^{n/2}}
        \exp\left(-\frac{|x|^2}{2\sigma^2}\right)\dd x.
\]
Given any set $A\subset\R^n$, its polar is defined as
\[
        A^\circ=\{y\in\R^n:\ip{x}{y}\le1\text{ for all }x\in A\}.
\]
The polar set is always closed, convex, and contains the origin. For a  measurable set $A$ we consider the Gaussian volume product
\begin{equation}\label{eq:Gaussian product}
G^n_\sigma(A)=\gamma_\sigma^n(A)\,\gamma_\sigma^n(A^\circ)
\end{equation}
and ask the following question.

\begin{question}\label{q:main}
For fixed $\sigma>0$, which sets $A\subset\R^n$ maximize $G^n_\sigma(A)$?
\end{question}

The problem naturally reduces to the case when $A$ is already a closed convex set containing the origin, since for every set $A\subset\R^n$,
\[
        A^{\circ\circ}=\overline{\conv(A\cup\{0\})},
        \qquad (A^{\circ\circ})^\circ=A^\circ,
\]
and $A\subset A^{\circ\circ}$. Replacing $A$ by $A^{\circ\circ}$ can therefore only increase the first Gaussian factor and leaves the second unchanged. Thus it is enough to work on the following class of measurable sets
\[
        \class=\{K\subset\R^n:K\text{ is closed, convex, and }0\in K\}.
\]

A natural point of comparison is the classical Blaschke--Santal\'o inequality \cite{Blaschke, Santalo}. It asserts that for every convex body $K\subset\R^n$ there exists a point $z\in\operatorname{int}K$, the Santal\'o point of $K$, such that
\begin{equation*}
    \vol(K)\vol((K-z)^\circ)\le \vol(B_2^n)^2,
\end{equation*}
 where $B_2^n$ denotes the Euclidean unit ball. In the centrally symmetric case the Santal\'o point is the origin, and the inequality becomes
\[
        \vol(K)\vol(K^\circ)\le \vol(B_2^n)^2.
\]
The Blaschke--Santal\'o inequality has a long history and admits many formulations and proofs, reflecting its connections with affine isoperimetric inequalities, harmonic analysis, and transport-entropy theory. We refer to \cite{Schneider, AAGM,FMZ} for general background on the volume product, to \cite{Blaschke,Santalo,SaintRaymond,Ball-thesis,MeyerPajorBlaschke} for some of its proofs, and to \cite{BallLogConcave,AAKM,FM,BianchiKelly,FathiTalagrand,FGSZ} for functional variants. For completeness, let us mention that the lower bound for the volume product is much subtler. The symmetric Mahler conjecture, posed by Mahler \cite{Mahler1939}, predicts that the minimum among centrally symmetric convex bodies is attained by the cube; it is known in dimensions $2$ and $3$ \cite{Mahler1939,IriyehShibata}. The non-symmetric three-dimensional Mahler conjecture was recently proved in \cite{ChenLiXiXuMahler} and is open in higher dimensions.

Question \ref{q:main} is a Gaussian counterpart of the Santal\'o volume-product problem, but with a different flavour as we do not assume that the set is centered, nor translate the set by its Santal\'o point. The origin is fixed before polarity is taken, and we consider all the sets in the image of the polarity transform.

The Gaussian volume product was already considered by Cordero-Erausquin, who proved that among centrally symmetric convex bodies the Euclidean unit ball is optimal.

\begin{manualtheorem}{A}[Cordero-Erausquin \cite{CorderoSantalo}]\label{thm:symmetric}
Let $K\subset\R^n$ be a centrally symmetric convex body. Then, for all $\sigma>0$,
\[
        \gamma_\sigma^n(K)\,\gamma_\sigma^n(K^\circ)
        \le \gamma_\sigma^n(B_2^n)^2.
\]
\end{manualtheorem}

Cordero-Erausquin asked whether, more generally, for an even log-concave measure $\mu$ and a symmetric convex body $K$ one has
\[
        \mu(K)\,\mu(K^\circ)\le \mu(B_2^n)^2.
\]
This conjecture in full generality remains open, and is connected to the $B$-conjecture. It has been proved to be true for unconditional log-concave measures, see \cite{FMZ,MilmanYehudayoff}. \vspace*{2mm}

When central symmetry is dropped, a phase transition phenomenon in $\sigma$ appears. Intuitively, when $\sigma$ is small, most of the mass is near the origin and so the maximizer is allowed to have a small Lebesgue measure, but as $\sigma$ becomes large, the maximizer needs to capture more Lebesgue measure. More precisely, as $\sigma\to \infty$, clearly $\gamma_\sigma^n(B_2^n)=\gamma_1^n(B_2^n/\sigma)\to0$, while, for example, the orthant $K=\R_+^n$ satisfies $\gamma_\sigma^n(K)=2^{-n}=\gamma_\sigma^n(K^\circ)$ independently of $\sigma$. Thus a phase transition must occur. 

First, we show a complete characterization of the maximizer in dimension one.

\begin{theorem}\label{thm:one-dimensional}
Let $\sigma>0$. Among all measurable $K\subset\R$, the Gaussian volume product $G_\sigma^1(K)$
is maximized as follows. 
If $\sigma\le1$, the unique maximizer is the symmetric interval $[-1,1]$.
If $\sigma>1$, the maximizers are, up to reflection,
\[
        K_\sigma=[-a_\sigma^{-1},a_\sigma],
\]
where $a_\sigma>1$ is the unique solution of
\begin{equation}\label{eq:asigma}
        a_\sigma^2\exp\left(-\frac{a_\sigma^2-a_\sigma^{-2}}{2\sigma^2}\right)=1.
\end{equation}
\end{theorem}

In particular we see that for $\sigma\le1$, the symmetric interval remains optimal. For $\sigma>1$, however, the maximizer becomes asymmetric. 

Our second result shows that, although maximizers need not be centrally symmetric, they retain a strong form of symmetry. 

\begin{theorem}\label{thm:revolution}
 Let $n\ge2$, and let $K\subseteq \R^n$ be a maximizer of $G^n_\sigma$. Then $K$ is a body of revolution. 
\end{theorem}

This theorem reduces the higher-dimensional non-symmetric problem to a one-dimensional variational problem. Its proof uses Steiner symmetrization with respect to suitably chosen medial hyperplanes, an approach adapted from Meyer--Pajor \cite{MeyerPajorSantalo} going back to Ball \cite{Ball-thesis}.

Using a Wulff-shape differentiation argument, sometimes called Aleksandrov's lemma, we show that maximizers must be smooth. In doing so, we derive an Euler--Lagrange equation. 
While the Euclidean ball is always a critical point, namely it satisfies the Euler--Lagrange equation, it is not always a maximizer.

\begin{theorem}\label{thm:not-ball}
Let $n\ge2$ and $\sigma^2>{\frac{2}{n+1}}$, then $B_2^n$ is not a maximizer of $G^n_\sigma$.
\end{theorem}

On the other hand, for $\sigma$ sufficiently small, the unit Euclidean ball is indeed the unique maximizer. First, in dimension $2$, we show that $\sigma^2=\frac 2{3}$ is the exact threshold. 

\begin{theorem}\label{thm:small-sigma2}
Let $n=2$ and  $\sigma^2\le{ \frac23}$,
then $K=B_2^2$ is the unique maximizer of $G^2_\sigma$.  
\end{theorem}

Second, for $n\ge 3$, we prove the following.

\begin{theorem}\label{thm:small-sigma}
Let $n\ge 3$, and $\sigma^2\le\frac1{ n}$,
then $K=B_2^n$ is the unique maximizer of $G^n_\sigma$.  
\end{theorem}

 In dimension $n\ge 3$, the gap between \( {\frac{1}{n}}\) and \( {\frac{2}{n+1}}\) remains open.

\subsection*{Organization} 
In Section~2, we prove Theorem \ref{thm:one-dimensional} which resolves the problem completely in dimension one.  In Section~3, we establish the existence of maximizers and prove that every maximizer is a bounded convex body containing the origin in its interior. We then prove Theorem \ref{thm:revolution} showing that every maximizer is a body of revolution. Section~4 is devoted to the regularity of maximizers. Finally, in Section~5, we analyze the role of the Euclidean unit ball. We show that it cannot be a maximizer when $\sigma^2>\frac{2}{n+1}$ (Theorem \ref{thm:not-ball}), while it is the unique maximizer in the plane for $\sigma^2\leq \frac{2}{3}$ (Theorem \ref{thm:small-sigma2}) and, in dimensions $n\geq3$, for $\sigma^2\leq \frac{1}{n}$ (Theorem \ref{thm:small-sigma}).

\section*{Acknowledgments.}
The first named author was supported by the Israel Science Foundation (grant 1626/25). The second and third named authors thank LAMA at the University Gustave Eiffel for their hospitality. 
\paragraph{AI-assisted tools.}
The authors used an LLM as an AI-assisted tool during the preparation of this manuscript. This includes mainly checking computations and suggesting some arguments. All statements and proofs were verified by the authors, who take full responsibility for the content.

\section{The one-dimensional problem}\label{sec:one-dim}

In dimension one, a convex set which includes the origin is an interval $K=[-a,b]$, where $a,b\in[0,\infty]$ (we denote $[-a,\infty] = [-a, \infty)$). Then 
\[
        K^\circ=[-a^{-1},b^{-1}],
\]
where we use the convention $\frac{1}{0}=\infty$ and $\frac{1}{\infty}=0$.

Fix $\sigma>0$. Maximizing $G_\sigma^1$ is equivalent to finding the global maximum of the function 
\begin{equation}\label{eq:Gab}
        G(a,b)=
        \left(\int_{-a}^{b} e^{-t^2/(2\sigma^2)}\dd t\right)
        \left(\int_{-1/a}^{1/b} e^{-s^2/(2\sigma^2)}\dd s\right).
\end{equation}
Note the two underlying symmetries 
\begin{equation}\label{eq:G-symmetries}
        G(a,b)=G(b,a)\quad \text{ and } \quad G(a,b)=G(1/a,1/b),
\end{equation}
which will be used repeatedly. In particular, it follows that we may work in the closed triangle
\[
        D=\{(a,b):0\le a\le b\}.
\]
\begin{remark}\label{rem:compactification}
Note that we may compactify this domain by the change of variables \(a=\tan\alpha\), \(b=\tan\beta\), where \(0\le\alpha\le\beta\le\pi/2\). The function \(G\) extends continuously to this compact triangle, with the usual conventions \(1/0=\infty\) and \(1/\infty=0\). Hence a maximizer exists.
\end{remark}

\begin{proof}[Proof of Theorem \ref{thm:one-dimensional}]

Define
\[
        H(t)=\int_0^t e^{-u^2/(2\sigma^2)}\dd u,
\]
and note that on $D=\{(a,b):0\le a\le b\}$ we have
\[
        G(a,b)=\bigl(H(b)-H(-a)\bigr)\bigl(H(1/b)-H(-1/a)\bigr).
\]

We will now proceed to find the maximizers.

\subsection*{Case 1: Boundary points}
The finite boundary of $D$ consists of the two faces \(a=0\) and
\(a=b\). The remaining case \(b=\infty\)  is equivalent to \(a=0\), because of $G(a,\infty)=G(0,1/a)$. Consequently it is enough to rule out the face \(a=0\) and to analyze the
symmetric face \(a=b\).

First, we note that on the face $a=b$  the interval is centrally symmetric. Hence, by Theorem \ref{thm:symmetric} the maximum in this class is attained uniquely at $a=b=1$.  

It remains to show that the points $(0,b)$ cannot maximize $G$. To this end we will show that there is a small perturbation in $a$ that increases $G$. Indeed, letting \(\varphi(t)=e^{-t^2/(2\sigma^2)}\), and fixing \(b\in(0,\infty]\) for
\(\varepsilon>0\), we have
\[
\int_{-\varepsilon}^{b}\varphi(t)\,dt
=
\int_{0}^{b}\varphi(t)\,dt+\varepsilon+o(\varepsilon),
\]
whereas for the polar interval we have
\[
\int_{-1/\varepsilon}^{1/b}\varphi(s)\,ds
=
\int_{-\infty}^{1/b}\varphi(s)\,ds
-
\int_{-\infty}^{-1/\varepsilon}\varphi(s)\,ds.
\]
Since
\[
        \int_{-\infty}^{-1/\varepsilon}\varphi(s)\,ds=o(\varepsilon),
\]
we get
\[
        G(\varepsilon,b)-G(0,b)
        =
        \varepsilon
        \int_{-\infty}^{1/b}\varphi(s)\,ds
        +o(\varepsilon)>0
\]
for all sufficiently small \(\varepsilon>0\). Thus no point with $a=0$ and $b>0$ is maximal. The ``corner" $(a,b)=(0,0)$ gives product zero and is not maximal.

Thus every maximizer is either the symmetric candidate \((1,1)\), or an
interior critical point of \(G\) in the region \(0<a<b<\infty\).

\subsection*{Case 2: Interior critical points}

Suppose that $(a,b)\in {\rm int}(D)$ is an interior critical point. Differentiating \eqref{eq:Gab} gives
\[
\begin{aligned}
        \frac{\partial G}{\partial a}(a,b)
        &=H'(-a)\bigl(H(1/b)-H(-1/a)\bigr)
        -\bigl(H(b)-H(-a)\bigr)H'(-1/a)a^{-2}, \\
        \frac{\partial G}{\partial b}(a,b)
        &=H'(b)\bigl(H(1/b)-H(-1/a)\bigr)
        -\bigl(H(b)-H(-a)\bigr)H'(1/b)b^{-2}.
\end{aligned}
\]
Therefore, at a critical point, we must have
\[
        \frac{H'(-a)a^2}{H'(-1/a)}
        =
        \frac{H'(b)b^2}{H'(1/b)}
        =
        \frac{H(b)-H(-a)}{H(1/b)-H(-1/a)}.
\]
Using that $H'(t)=e^{-t^2/(2\sigma^2)}$ and defining
\begin{equation*}\label{eq:F-one-dim}
        f_\sigma(x)=x\exp\left(-\frac{x-x^{-1}}{2\sigma^2}\right),\qquad x>0,
\end{equation*}
at every interior critical point, we must have
\begin{equation}\label{eq:crit-h}
        f_\sigma(a^2)=f_\sigma(b^2)=\frac{H(b)-H(-a)}{H(1/b)-H(-1/a)}
        =\frac{\gamma^1_\sigma([-a,b])}{\gamma^1_\sigma([-1/a,1/b])}:=h(a,b).
\end{equation}
Moreover, note that
\begin{equation}\label{eq:logF-derivative}
        (\log f_\sigma)'(x)
        =\frac1x-\frac1{2\sigma^2}\left(1+\frac1{x^2}\right)
        =-\frac{(x-\sigma^2)^2+1-\sigma^4}{2\sigma^2x^2}.
\end{equation}

\subsection{The case \texorpdfstring{$\sigma\le1$}{sigma <= 1}}

If $\sigma\le1$, then \eqref{eq:logF-derivative} shows that the derivative $(\log f_\sigma)' <0$, and since $f_\sigma(x)> 0$ for all $x>0$, we get that $f_\sigma$ is strictly decreasing on $(0,\infty)$. Thus \eqref{eq:crit-h} forces $a=b$, which is not an interior point.   Together with the boundary analysis, this proves the first part of Theorem \ref{thm:one-dimensional}.

\subsection{The case \texorpdfstring{$\sigma>1$}{sigma > 1}}

Assume now that $\sigma>1$. The derivative $(\log f_\sigma)'$ changes sign at
\[
        x_- = \sigma^2-\sqrt{\sigma^4-1},
        \qquad
        x_+ = \sigma^2+\sqrt{\sigma^4-1},
        \qquad x_-=\frac{1}{x_+}.
\]
Thus $f_\sigma$ decreases on $(0,x_-)$, increases on $(x_-,x_+)$, and decreases on $(x_+,\infty)$. Since $f_\sigma(1)=1$ and $1\in(x_-,x_+)$, there is a unique $a_\sigma>1$ satisfying $f_\sigma(a_\sigma^2)=1$, which is exactly \eqref{eq:asigma}.
Clearly, we also have that $f_\sigma(a_\sigma^{-2})=1$, because $f_\sigma(1/x)=1/f_\sigma(x)$.

We next show that $(a,b)=(a_\s^{-1},a_\s)$ are the only interior critical points, up to reflection. Let $(a,b)$ be an interior critical point with $a\le b$. Denote the common value in \eqref{eq:crit-h} by $c$.

{\bf 2.2.i) First suppose $c>1$.} Apart from the limiting case $c=f_\sigma(x_+)$, which follows by continuity from the same argument, the equation $f_\sigma(x)=c$ has three roots $y_1<y_2<y_3$, with
\[
        y_1<x_-<1<y_2<x_+<y_3.
\]
Then the possible pairs $(a^2,b^2)$ are $(y_1,y_2)$, $(y_1,y_3)$, and $(y_2,y_3)$.

For the pairs $(a^2,b^2) \in \{(y_1,y_2),(y_1,y_3)\}$ we have $ab<1$. Indeed, $y_1<1/y_2$ follows from $y_1<x_-=1/x_+<1/y_2$. Also $1/y_3<x_-$ and $f_\sigma(1/y_3)=1/c<f_\sigma(y_1)$; since $f_\sigma$ is strictly decreasing on $(0,x_-)$, this gives $y_1<1/y_3$. Thus
\[
        -[-a,b]=[-b,a]\subsetneq[-1/a,1/b]=[-a,b]^\circ.
\]
Using this inclusion and the evenness of the Gaussian density, $\gamma_\sigma^1([-a,b])<\gamma_\sigma^1([-1/a,1/b])$, that is, $h(a,b)<1<c$, contradicting \eqref{eq:crit-h}.

It remains to rule out $(a^2,b^2)=(y_2,y_3)$. For fixed $a_0>0$, the function $b\mapsto h(a_0,b)$ is strictly increasing, and the same is true in the $a$ variable. 
Hence, we have that 
\[ 1=h(1,1)<h(a,1)<h(a,a).
\]
Consider the function $F(a)=h(a,a)-f_\sigma(a^2)$. Since $f_\sigma(1)=1$, we have that  $F(1)=1-1=0$. Moreover, we claim that $F(a)>0$ for all $a>1$. This implies that for $a<b$, since $f_\sigma(b^2)= f_\sigma(a^2) <h(a,a)<h(a,b)$, we cannot have equality, hence $(y_2,y_3)$ cannot be a solution to the equation  \eqref{eq:crit-h}.

Let $\eta(a)=G(a,a)$. Computing its derivative, we see that $F(a)>0$ if and only if $\eta'(a)<0$. By the (B)-theorem, $t\mapsto \eta(e^t)$ is log-concave on $\R$ and  since it is even, it is maximal at $t=0$. This implies that $\eta$ is increasing on $(0,1)$ and decreasing on $(1,+\infty)$. In particular, the only solution to $F(a)=0$ for $a\ge1$ is $a=1$. Therefore, 
$F(a)>0$ for all $a>1$.

{\bf 2.2.ii) Suppose $c<1$. }  Since the polar symmetry $(a,b)\mapsto(1/a,1/b)$ sends a critical point with common value $c$ to a critical point with common value $1/c$, it also excludes all critical points with $c<1$.

{\bf 2.2.iii) Suppose $c=1$.} It remains to examine $c=1$. The roots of $f_\sigma(x)=1$ are
\[
        a_\sigma^{-2},\qquad 1,
        \qquad a_\sigma^2.
\]
The condition $h(a,b)=1$ eliminates all pairs except $(a,b)=(1,1)$ and, up to reflection, $(a,b)=(a_\sigma^{-1},a_\sigma)$.

 To compare the two, we work as above, define 
\[ I(a) = \int_{-1/a}^a e^{-x^2/2\sigma^2}dx,\]
then $G(1/a,a)=I^2(a)$ and
\[ I'(a) = e^{-a^2/2\sigma^2}  - e^{-(1/a^2)/2\sigma^2}a^{-2}.\]
We see that $I'(a)>0$ if and only if $f_\s(a^2)>1$ which means, as we did when studying $f_\s$, that $I(a)$ is increasing on $(0,1/a_\sigma)$, decreasing on $(1/a_\sigma, 1)$, increasing on $(1,a_\sigma)$ and decreasing again on $(a_\s, \infty)$. Hence $I$ is maximized at $a=a_\sigma$, and this value is strictly larger than $I(1)$. This completes the proof of Theorem \ref{thm:one-dimensional}. \end{proof}

\section{The structure of the maximizers: first observations}

\subsection{Existence of maximizers}\label{sec:existence}
The existence follows from the compactness of the family $\class$,
equipped with the local Hausdorff topology, and the continuity of the
functional. We include the details for completeness.

\begin{proposition}\label{prop:existence}
For every $\sigma>0$ and every $n\ge1$, the supremum of $G^n_\sigma(A)$ over all subsets $A\subset\R^n$ is attained by some $K\in \class$.
\end{proposition}

\begin{proof}
We use the local Hausdorff topology on closed convex sets.
Every sequence in $\class$ has a locally Hausdorff convergent
subsequence. Indeed, the common point $0$ rules out escape to the
horizon, so the subsequence-extraction theorem for set convergence
applies; the limit remains closed and convex and contains $0$.
For closed convex sets, set convergence is equivalent to Hausdorff
convergence of bounded truncations. It follows
that $\class$ is compact in the local Hausdorff topology; see
\cite[Chapter 4]{RockafellarWets}. % \cite[Proposition~4.15, Exercise~4.16, and Theorems~4.18 and~4.42]{RockafellarWets}

Polarity is continuous in this topology; see
\cite[Theorem~7.2]{Wijsman1966}.
%Polarity is continuous on this class for the local Hausdorff topology. 
Moreover, if $K_j\to K$ locally in the Hausdorff sense, then
\[
        \gamma_\sigma^n(K_j)\to\gamma_\sigma^n(K).
\]
Indeed, after intersecting with a large Euclidean ball, Hausdorff convergence of convex sets implies pointwise convergence of the indicators away from the boundary of $K$; the boundary of a proper closed convex set has Lebesgue, hence Gaussian, measure zero, and the Gaussian tail outside the large ball is uniformly small. The cases $K=\R^n$ and lower-dimensional $K$ are handled by the same argument, with the evident interpretations. Applying the same reasoning to the polars gives
\[
        \gamma_\sigma^n(K_j^\circ)\to\gamma_\sigma^n(K^\circ).
\]
Thus, $G^n_\sigma$ is continuous on the compact space $\class$, and therefore it attains its maximum.
\end{proof}

\subsection{Boundedness of maximizers}\label{sec:bddness}

It turns out that the maximizers must be bounded and include the origin in their interior. This fact  utilizes the fast decay of the measure  $\gamma_\sigma^n$.  

\begin{proposition}\label{lem:maximizer-is-body}
Let $K\in\class$ be a maximizer of $G^n_\sigma$.  Then both $K$ and
$K^\circ$ are bounded convex sets and
\[
 0\in\operatorname{int}K\cap\operatorname{int}K^\circ.
\]
\end{proposition}

\begin{proof}
Since $G^n_\sigma(B_2^n)>0$, maximality gives
$\gamma_\sigma^n(K)>0$ and $\gamma_\sigma^n(K^\circ)>0$.
We prove that the origin is an interior point of $K$.  The same argument can
then be applied to $K^\circ$, which is also a maximizer because
$K^{\circ\circ}=K$. This will also imply boundedness since the dual of a set which includes the origin in its interior is bounded. 

Suppose, towards a contradiction, that $0\in\partial K$.  Choose $u\in S^{n-1}$ so that
\[
 K\subset\{x:\langle x,u\rangle\geq 0\}.
\]
For $\varepsilon>0$, set
\[
 K_\varepsilon
 =\overline{\operatorname{conv}}\bigl(K\cup\varepsilon B_2^n\bigr).
\]
Then
\[
 K_\varepsilon^\circ
 =K^\circ\cap\varepsilon^{-1}B_2^n.
\]
Writing
\[
 \delta_\varepsilon
 =\gamma_\sigma^n(K_\varepsilon)-\gamma_\sigma^n(K)>0,
 \qquad
 \eta_\varepsilon
 =\gamma_\sigma^n(K^\circ)-\gamma_\sigma^n(K_\varepsilon^\circ)>0,
\]
we have
\[
 \delta_\varepsilon
 \geq \gamma_\sigma^n\bigl(\varepsilon B_2^n
       \cap\{x:\langle x,u\rangle<0\}\bigr)
 =\frac12\gamma_\sigma^n(\varepsilon B_2^n)
 \geq c_{n,\sigma}\varepsilon^n
\]
for sufficiently small $\varepsilon$, whereas
\[
 0\leq\eta_\varepsilon
 \leq\gamma_\sigma^n\bigl(\mathbb R^n\setminus
       \varepsilon^{-1}B_2^n\bigr)
 =c'_{n,\sigma}\int_{1/\varepsilon}^{\infty}
r^{n-1}e^{-r^2/(2\sigma^2)} dr = o(\varepsilon^n).
\]
We see that 
\[
 \begin{aligned}
 G^n_\sigma(K_\varepsilon)-G^n_\sigma(K) %= 
%\gamma_\sigma^n(K_\eps^\circ)\gamma_\sigma^n(K_\eps) - 
%\gamma_\sigma^n(K^\circ)\gamma_\sigma^n(K)
 & = 
  (\gamma_\sigma^n(K^\circ)-\eta_\eps)(\gamma_\sigma^n(K)+\delta_\eps) - 
\gamma_\sigma^n(K^\circ)\gamma_\sigma^n(K)\\
 &\ge \gamma_\sigma^n(K^\circ)\delta_\varepsilon-%(\gamma_\sigma^n(K)+\delta_\varepsilon)
 \eta_\varepsilon \geq \gamma_\sigma^n(K^\circ)c_{n,\sigma}\varepsilon^n-o(\varepsilon^n)
 \end{aligned}
\]
so that for small $\varepsilon$ this quantity is positive, which is a contradiction.  Hence $0\in\operatorname{int}K$, which also implies that $K^\circ$ is bounded. Applying the same argument to $K^\circ$ completes the proof. 
\end{proof}

\subsection{Maximizers are bodies of revolution}\label{sec:symmetrization}

In this section, we show that one can symmetrize a given set with respect to certain medial hyperplanes while increasing the Gaussian volume product. This is an adaptation of the approach of Meyer and Pajor \cite{MeyerPajorBlaschke}, originating in Ball \cite{Ball-thesis}. This procedure can then be iterated, but unlike in the classical case, not in all directions, so that it produces, in the limit, a body a revolution. 

First, we recall the definition of  Steiner symmetrization with respect to a hyperplane $H=\theta^\perp$ containing the origin, and show it increases the Gaussian measure of the set. Let $P_H$ be the orthogonal projection onto $H$. For a closed convex set $K\subset\R^n$, define the fiber
\[
        K(x)=\{t\in\R:x+t\theta\in K\},\qquad x\in P_HK.
\]
The Steiner symmetral of $K$ with respect to $H$ is
\[
        S_HK=\left\{x+t\theta:x\in P_HK,\ |t|\le \frac{|K(x)|}{2}\right\},
\]
with the convention that a fiber of infinite length is replaced by the whole line.
\begin{lemma}\label{lem:gauss-steiner}
For every closed convex set $K\subset\R^n$ with non-empty interior and every hyperplane $H$ containing the origin,
\[
        \gamma_\sigma^n(S_HK)\ge\gamma_\sigma^n(K).
\]
Moreover, for a closed convex body $K$, equality implies $S_HK=K$.
\end{lemma}

\begin{proof}
Without loss of generality, we may rotate $\R^n$, so that $H=e_n^\perp$.  By Fubini, and the fact that $\gamma_\sigma^n$ is a product measure we have
\[
        \gamma_\sigma^n(K)=\int_{P_HK}\gamma_\sigma^1(K(x))\dd\gamma_\sigma^{n-1}(x),
\]
whereas
\[
        \gamma_\sigma^n(S_HK)=\int_{P_HK}\gamma_\sigma^1\left(\left[-\frac{|K(x)|}{2},\frac{|K(x)|}{2}\right]\right)\dd\gamma_\sigma^{n-1}(x).
\]
Since $\gamma_\sigma^1$ is log-concave and symmetric,  among intervals of a fixed length, the centered interval has maximal one-dimensional Gaussian measure. The inequality follows by integrating this one-dimensional inequality over $P_HK$.

The one-dimensional inequality is strict unless the interval is already centered. Hence equality in the integral forces almost every fiber to be centered, which for a closed convex set must hold for all fibers. Thus $S_HK=K$.
\end{proof}

Next, we show that the Gaussian measure of the polar also increases under Steiner symmetrization. For this, one needs to assume that the hyperplane $H$ is \emph{medial} to $K^\circ$. If $H=\theta^\perp$, set
\[
        H^+=\{x\in\R^n:\ip{x}{\theta}>0\},\qquad
        H^-=\{x\in\R^n:\ip{x}{\theta}<0\}.
\]
We say that $H$ is medial for a measurable set $L$ if
\[
        \gamma_\sigma^n(L\cap H^+)=\gamma_\sigma^n(L\cap H^-).
\]
We use the following lemma of Ball \cite{Ball-thesis}, in the form used by Meyer and Pajor.
\begin{lemma}[Ball \cite{Ball-thesis}]\label{lem:Ball}
Let $f,g,h:(0,\infty)\to[0,\infty)$ be integrable functions such that, for every $s,t>0$,
\[
        h\left(\frac{2st}{s+t}\right)
        \ge f(s)^{t/(s+t)}g(t)^{s/(s+t)}.
\]
Then
\[
        \left(\int_0^\infty h(u)\dd u\right)^{-1}
        \le
        \frac12\left(\int_0^\infty f(s)\dd s\right)^{-1}
        +
        \frac12\left(\int_0^\infty g(t)\dd t\right)^{-1},
\]
with the usual convention if one of the integrals is zero.
\end{lemma}
\begin{proposition}\label{prop:dual-steiner}
Let $K\subset\R^n$ be closed and convex, and assume that $H$ is medial for $K^\circ$. Then
\[
        \gamma_\sigma^n((S_HK)^\circ)\ge \gamma_\sigma^n(K^\circ).
\]
\end{proposition}

\begin{proof}
 For $t\in\R$, define
\[
        K^\circ(t)=\{x\in H:x+t\theta\in K^\circ\}.
\]
Meyer and Pajor \cite{MeyerPajorBlaschke} showed that, for any $s,t>0$,
\begin{equation}\label{eq:MP-inclusion}
        \frac{t}{s+t}K^\circ(s)+\frac{s}{s+t}K^\circ(-t)
        \subset
        (S_HK)^\circ\left(\frac{2st}{s+t}\right).
\end{equation}
By log-concavity of Gaussian measure on $H$, \eqref{eq:MP-inclusion} implies
\[
\begin{aligned}
\gamma_\sigma^{n-1}\left((S_HK)^\circ\left(\frac{2st}{s+t}\right)\right)        \ge
\gamma_\sigma^{n-1}(K^\circ(s))^{\frac{t}{s+t}}
\gamma_\sigma^{n-1}(K^\circ(-t))^{\frac{s}{s+t}}.
\end{aligned}
\]
Omitting the normalizing constant, for $s,t,u>0$, define 
\[
\begin{aligned}
        f(s)&=e^{-\frac{s^2}{2\sigma^2}}\gamma_\sigma^{n-1}(K^\circ(s)), \ \ 
        g(t)&=e^{-\frac{t^2}{2\sigma^2}}\gamma_\sigma^{n-1}(K^\circ(-t)), \ \ 
        h(u)&=e^{-\frac{u^2}{2\sigma^2}}\gamma_\sigma^{n-1}((S_HK)^\circ(u)).
\end{aligned}
\]
Therefore, for any $s,t>0$ we have
\[
        h\left(\frac{2st}{s+t}\right)
        \ge f(s)^{\frac{t}{s+t}}g(t)^{\frac{s}{s+t}},
        \qquad s,t>0.
\]
By Lemma \ref{lem:Ball}, this implies
\[
        \left(\int_0^\infty h(u)\dd u\right)^{-1}
        \le
        \frac12\left(\int_0^\infty f(s)\dd s\right)^{-1}
        +
        \frac12\left(\int_0^\infty g(t)\dd t\right)^{-1},
\]
which amounts to \[
\frac{1}{\gamma_\s^n((S_HK)^\circ\cap H^+)}\le\frac{1}{2}\left(\frac{1}{\gamma_\s^n(K^\circ\cap H^+)}+\frac{1}{\gamma_\s^n(K^\circ\cap H^-)}\right).
\]
We assumed that $\gamma_\s^n(K^\circ\cap H^+)=\gamma_\s^n(K^\circ\cap H^-)$. Moreover, since $S_HK$ is symmetric with respect to $H$ its polar is also symmetric with respect to $H$, and we get that $\gamma_\s^n((S_HK)^\circ \cap\, H^+)= \frac{1}{2}\gamma_\s^n((S_HK)^\circ)$. This gives the desired inequality.
\end{proof}

We will use the following lemma to construct a sequence of medial hyperplanes.

\begin{lemma}\label{lem:medial-containing-subspace}
Let $0\in L\subset\R^n$ be measurable and let $E\subset\R^n$ be a subspace of codimension two. Then there exists a hyperplane $H$ containing $E$ that is medial for $L$.
\end{lemma}

\begin{proof}
If $\gamma_\sigma^n(L)=0$, then any hyperplane containing $E$ is medial for $L$. Thus assume that $\gamma_\sigma^n(L)>0$.

Let $S=S^{n-1}\cap E^\perp$. Since ${\rm dim}(E^\perp)=2$, we have $S\cong S^1$ is a 1-dimensional sphere. For every $u\in S$ define $H(u) =u^\perp$, and note that $E\subset H(u)$. Let $H_+(u)=\{x\in \R^n:\ip{x}{u}>0\}$, and define the function
\[
        f(u)=\frac{\gamma_\sigma^n(L\cap H_+(u))}{\gamma_\sigma^n(L)}.
\]
Note that the function $f$ is continuous, and because Gaussian measure gives zero mass to hyperplanes, $f(-u)=1-f(u)$. The intermediate value theorem implies that there exists $u_0\in S$ with $f(u_0)=\frac{1}{2}$. Thus $H(u_0)$ is medial for $L$, as claimed.
\end{proof}

%We are now ready to prove our main theorem.

\begin{proof}[Proof of Theorem \ref{thm:revolution}]
We first show that, starting from any \(K\in\class\), one can construct, using Steiner symmetrizations, a set of revolution \(\widetilde K\in\class\) such that
\[
     G^n_\sigma(\widetilde K)\ge G^n_\sigma(K).
\]
It follows that if $K$ is a maximizer of \eqref{eq:Gaussian product}, then since the construction gives a sequence of sets with nondecreasing Gaussian volume products, we must have equality at every symmetrization step.  We then apply the equality statement in Lemma \ref{lem:gauss-steiner} to each step, in order to conclude that $K$ itself must be a set of revolution. 

Put \(K_0=K\). We construct orthonormal vectors
\(u_1,\ldots,u_{n-1}\) recursively. Choose any vector $u_1$ such that the hyperplane $u_1^\perp$ is medial for $K^\circ$. Let $K_1=S_{u_1^\perp}K$. Then Lemma \ref{lem:gauss-steiner} and Proposition \ref{prop:dual-steiner} give
\[
        \gamma_\sigma^n(K_1)\ge \gamma_\sigma^n(K),
        \qquad
        \gamma_\sigma^n((K_1)^\circ)
        \ge \gamma_\sigma^n(K^\circ).
\]
Assume we have found orthonormal vectors $u_1,\ldots,u_k$ for $1\le k\le n-2$ and we let $K_k=S_{u_k^\perp}(K_{k-1})$. 
Choose $u_{k+1}$ to be orthogonal to all $u_1, \ldots, u_k$ and such that $u_{k+1}^\perp$ is medial to $K_k^\circ$. This is possible since $E_k=\operatorname{span}(u_1,\ldots,u_k)$ has codimension at least $2$, so Lemma \ref{lem:medial-containing-subspace} gives a hyperplane $H_{k+1}=u_{k+1}^\perp$ containing $E_k$ that is medial for $K_k^\circ$.

Applying Lemma \ref{lem:gauss-steiner} and Proposition \ref{prop:dual-steiner} as above, we get $S_{H_{k+1}}K_k=K_{k+1}$ has Gaussian volume product at least as large as $K$ and all $K_i$ for $i=1,\ldots,k$. More generally, both $\gamma^n_\sigma(K_k)$ and $\gamma^n_{\sigma}(K_k^{\circ})$ are non-decreasing in $k$. This completes the induction.

After $n-1$ steps, the normals $u_1,\ldots,u_{n-1}$ span a hyperplane $E$. Let $e$ be a unit vector orthogonal to $E$. Without loss of generality, 
after applying a rotation, we may assume that
\(u_i=e_i\) for \(i=1,\ldots,n-1\), and that the common axis is
\(\mathbb Re_n\). Letting $\hat K=K_{n-1}$, we note that it is invariant with respect to reflections about $e_1^\perp, \ldots, e_{n-1}^\perp$. The same is true for the dual body, which has the same set of \((n-1)\) symmetries.

Since the restricted Gaussian density on the slice is even, every hyperplane containing the \(e_n\)-axis bisects the Gaussian measure of each fiber. Consequently, every hyperplane $H$ containing the $e_n$-axis is medial for $\hat K^\circ$.

This property (of having centrally symmetric sections in direction $e_n$) 
persists when we make any other Steiner symmetrization of $\hat K$ with respect to some $H$ which includes $e_n$, since the symmetrization works separately on these fibers. 

Next, choose a dense sequence \((u_j)_{j\ge n}\subset S^{n-1}\cap e_n^\perp\) (suitable for convergence of successive Steiner symmetrizations). Let $\widetilde{K}_{n} = \hat K$ and construct a sequence of bodies $\widetilde{K}_{j+1} = S_{u_{j+1}^\perp} \widetilde{K}_{j}$ for $j = n,n+1, \ldots$, such that 
both $\gamma^n_\sigma(\widetilde{K}_j)$ and $\gamma^n_{\sigma}(\widetilde{K}_j^{\circ})$ are again increasing in $j$.

By the standard convergence theorem for Steiner symmetrization in the local Hausdorff topology, the iterates converge to a convex
set \(\widetilde K\) whose sections perpendicular to \(e_n\) are Euclidean
balls (of various radii) centered on the \(e_n\)-axis. This means that $\widetilde{K}$ is a body of revolution.
Moreover,  $\gamma^n_\sigma(\widetilde{K})\ge \gamma^n_\sigma({K})$ and $\gamma^n_{\sigma}(\widetilde{K}^{\circ})\ge \gamma^n_{\sigma}(K^{\circ})$, completing the proof.
\end{proof}

\section{$C^\infty$ Regularity of maximizers and the Euler Lagrange equation}\label{sec:profiles}

We will use a relatively standard differentiation argument to prove that the maximizers have smooth support functions.
 Recall that we have already shown that for $n\ge 2$ every maximizer of \eqref{eq:Gaussian product} is a convex body of revolution which includes the origin in its interior. Hence, after a suitable rotation, we can write it as
\[
        K_f=\{(x,t)\in\R^{n-1}\times\R: |x|\le f(t)\},
\]
where $f$ is a non-negative concave function on its support interval $I = [-a, b]$ with $0<a,b$. 
Further, the polar of a set of revolution is again a set of revolution. A direct calculation gives that $K_f^\circ=K_g$, where
\begin{equation}\label{eq:T-transform}
        g(s)=Tf(s)=\inf_{\{t:f(t)>0\}}\frac{1-st}{f(t)},\qquad s\in I^\circ.
\end{equation}
Therefore, in higher dimensions, Question \ref{q:main} reduces to the one-dimensional profile problem
\begin{equation}\label{eq:reduced-F}
        G^n_\sigma(f)=\gamma_\sigma^n(K_f)\,\gamma_\sigma^n(K_{Tf}).
\end{equation}
We next show that the maximizers of $G^n_\sigma$ are smooth and satisfy a certain Euler-Lagrange equation. In what follows, we use the support function of \(K\) rather than its profile function, as it is more natural for Wulff perturbations and allows us to use the regularity theory for the Gaussian Minkowski problem.

\begin{proposition}\label{prop:regularity-maximizer}
Let \(n\ge 2\), let \(\sigma>0\), and let
\(K\subset\mathbb R^n\) be a maximizer of \(G_\sigma^n\).
Then $h_K\in C^\infty( S^{n-1})$ and \(K\) has everywhere positive Gauss curvature; equivalently, $ \nabla_S^2h_K+h_KI>0$ on $S^{n-1}$. Moreover, \(h_K\) satisfies \vspace{2mm}
\[
\begin{aligned}
    \gamma_\sigma^n(K^\circ)
    \exp\left(
        -\frac{h_K^2+|\nabla_Sh_K|^2}{2\sigma^2}
    \right)
    \det\left(\nabla_S^2h_K+h_KI\right)
    &=
    \gamma_\sigma^n(K)
    \exp\left(
        -\frac{1}{2\sigma^2h_K^2}
    \right)
    h_K^{-(n+1)}
\end{aligned}
\]
on \( S^{n-1}\).
\end{proposition}

\begin{proof}
By Proposition \ref{lem:maximizer-is-body}, \(K\) is a convex
body and \(0\in\operatorname{int}K\). Hence, writing
\(h=h_K\), there exist constants \(m,M>0\) such that
\[
    0<m\le h(u)\le M
    \qquad
    \text{for every }u\in S^{n-1}.
\]
Let $c_{n,\sigma}=(2\pi\sigma^2)^{-n/2}$ and consider the Gaussian surface area measure
\(S_{\gamma_\sigma,K}\) (see \cite{huang2021minkowski} where several of the computations are very similar to the ones we use below). It is characterised by
\begin{align}\label{eq:GaussSurfaceAreaMeasure}
     \int_{ S^{n-1}}\psi(u)\,
        dS_{\gamma_\sigma,K}(u)
    =
    c_{n,\sigma}
    \int_{\partial K}
        \psi(\nu_K(x))
        \exp\left(-\frac{|x|^2}{2\sigma^2}\right)
        d\mathcal H^{n-1}(x)
\end{align}
for every \(\psi\in C( S^{n-1})\). Here \(\mathcal H^{n-1}\) is the $(n-1)$-dimensional Hausdorff measure on $\R^{n}$, that is, the usual surface-area measure on \(\partial K\), and
\(\nu_K(x)\) denotes the outer unit normal, which is defined
for \(\mathcal H^{n-1}\)-almost every \(x\in\partial K\).

Let \(\varphi\in C(S^{n-1})\) be arbitrary. For
\(|t|\) sufficiently small, so that $h+t\varphi\ge \frac m2>0$, define the Wulff perturbation
\[
    K_t
    =
    [h+t\varphi]
    :=
    \bigcap_{u\in S^{n-1}}
    \left\{
       x\in\mathbb R^n:
       \langle x,u\rangle\le h(u)+t\varphi(u)
    \right\}.
\]
 The convex body \(K_t\) contains the origin
in its interior. By the first variation under Wulff
perturbations, by \cite[Theorem~3.3]{huang2021minkowski}, applied after
dilation to the standard Gaussian measure, we have
\begin{equation}
\label{eq:first-variation-K}
    \gamma_\sigma^n(K_t)
    =
    \gamma_\sigma^n(K) +  t\int_{ S^{n-1}}
        \varphi\,dS_{\gamma_\sigma,K}
    +o(t).
\end{equation}
We next obtain a one-sided first-order estimate for the polar. Since the support function of a Wulff shape satisfies $h_{K_t}\le h+t\varphi$,
we have
\[
    \rho_{K_t^\circ}(u)
    =
    \frac{1}{h_{K_t}(u)}
    \ge
    \frac{1}{h(u)+t\varphi(u)}.
\]
Therefore, using polar coordinates,
\[
\begin{aligned}
    \gamma_\sigma^n(K_t^\circ)
    &\ge
    c_{n,\sigma}
    \int_{ S^{n-1}}
    \int_0^{1/(h+t\varphi)}
        \exp\left(-\frac{r^2}{2\sigma^2}\right)
        r^{n-1}\,dr\,d\omega.
\end{aligned}
\]

Define, for \(s>0\), $F(s) = \int_0^{1/s}
        \exp\left(-\frac{r^2}{2\sigma^2}\right)
        r^{n-1}\,dr$, and since \(h\) is bounded above and below by positive constants and \(\varphi\) is bounded, we get 
\[
\begin{aligned}
    F(h+t\varphi)
    &=
    F(h)
    -
    t\varphi
    \exp\left(-\frac{1}{2\sigma^2h^2}\right)
    h^{-(n+1)}
    +o(t)
\end{aligned}
\]
uniformly on \( S^{n-1}\).

Define  $d\mu_h(u)=c_{n,\sigma}
    \exp\left(-\frac{1}{2\sigma^2h(u)^2}\right) h(u)^{-(n+1)}\,d\omega(u)$, and  note that the preceding expansion gives
\begin{equation}
\label{eq:first-variation-polar}
    \gamma_\sigma^n(K_t^\circ)
    \ge \gamma_\sigma^n(K^\circ)- t\int_{ S^{n-1}}\varphi\,d\mu_h
    +o(t).
\end{equation}

Since \(K\) is a maximizer, combining with the bounds
\eqref{eq:first-variation-K} and
\eqref{eq:first-variation-polar}, we obtain for all positive and negative \(t\) sufficiently small that
\[
    \bigl(\gamma_\sigma^n(K)+t\int_{ S^{n-1}} \varphi\,dS_{\gamma_\sigma,K}+o(t)\bigr)
    \bigl(\gamma_\sigma^n(K^\circ)-t\int_{ S^{n-1}}
        \varphi\,d\mu_h+o(t)\bigr)
    \le \gamma_\sigma^n(K)\gamma_\sigma^n(K^\circ).
\]
Expanding the product and letting \(t\downarrow0\) yields
\[
    \gamma_\sigma^n(K^\circ) \int_{ S^{n-1}}  \varphi\,dS_{\gamma_\sigma,K} \le \gamma_\sigma^n(K)\int_{ S^{n-1}}   \varphi\,d\mu_h,
\]
whereas letting \(t\uparrow0\) yields the reverse inequality.
Consequently,
\[
    \gamma_\sigma^n(K^\circ) \int_{S^{n-1}}  \varphi\,dS_{\gamma_\sigma,K} = \gamma_\sigma^n(K)\int_{ S^{n-1}}   \varphi\,d\mu_h,
\]
for every \(\varphi\in C( S^{n-1})\). Hence
\begin{equation}
\label{eq:weak-Euler}
      dS_{\gamma_\sigma,K}
    =
      \frac{\gamma_\sigma^n(K)}{\gamma_\sigma^n(K^\circ) }  c_{n,\sigma}
    \exp\left(-\frac{1}{2\sigma^2h^2}\right)
    h^{-(n+1)}\,d\omega.
\end{equation}

In particular, \(S_{\gamma_\sigma,K}\) is absolutely continuous with density 
bounded above and below by positive constants. We now apply the regularity theory for the Gaussian Minkowski problem. The \(p=1\) case of
\cite[Theorem 3.1]{feng2023lp} implies that
\(\partial K\) is \(C^1\) and strictly convex, and 
\[
    h\in C^{1,\beta}( S^{n-1})
\]
for every \(\beta\in(0,1)\). The theorem is stated for
the standard Gaussian measure, but the version with covariance
\(\sigma^2I\) follows by dilation.

In particular, the boundary point of $K$ with outer
normal \(u\) is given by
\[
        x(u)=h(u)u+\nabla_Sh(u)
\]
and depends Hölder-continuously on \(u\). Hence, the measure identity \eqref{eq:GaussSurfaceAreaMeasure} can now be rewritten as an ordinary surface area
equation:
\[
        dS_{K}(u)
        =
        \frac{\gamma_\sigma^n(K)}{\gamma_\sigma^n(K^\circ)} 
        \exp\left( \frac{|x(u)|^2}{2\sigma^2}\right)
        \exp\left(-\frac{h^{-2}(u)}{2\sigma^2} \right)
        h(u)^{-(n+1)}
        d\omega(u).
\]
Since
\[
        |x(u)|^2=h(u)^2+|\nabla_Sh(u)|^2,
\]
we obtain
\[
        dS_{K}(u) =
        \frac{\gamma_\sigma^n(K)}{\gamma_\sigma^n(K^\circ)}
        h(u)^{-(n+1)}
        \exp\left[
        \frac{1}{2\sigma^2}
        \left(
        h(u)^2+|\nabla_Sh(u)|^2-h(u)^{-2}
        \right)
        \right] d\omega (u).
\]

To show the improved regularity of $h$, we will use \cite[Theorem~4.5]{trudinger2008monge}, which summarizes classical results on existence, uniqueness and regularity for the Minkowski problem. 
Writing \(dS_K=f\,d\omega\), since \(h\in C^{1,\beta}\),  we get that
\(f\in C^{0,\beta}( S^{n-1})\) and, clearly, \(f>0\). Hence
\(\kappa:=1/f\) is also positive and belongs to
\(C^{0,\beta}( S^{n-1})\). Moreover,
\[
    \int_{ S^{n-1}}u\,\kappa(u)^{-1}\,d\omega(u)
    =
    \int_{ S^{n-1}}u\,dS_K(u)
    =0,
\]
where the last equality follows from the divergence theorem.
The Hölder regularity conclusion of
\cite[Theorem~4.5]{trudinger2008monge} therefore yields
\[
    h\in C^{2,\beta}( S^{n-1}).
\]
Thus the equation holds pointwise:
\[
        \det(\nabla_S^2h+hI)
        =
        \frac{\gamma_\sigma^n(K)}{\gamma_\sigma^n(K^\circ)}
        h^{-(n+1)}
        \exp\left[
        \frac{1}{2\sigma^2}
        \left(
        h^2+|\nabla_Sh|^2-h^{-2}
        \right)
        \right].
\]
The matrix \(\nabla_S^2h+hI\) is positive definite, because its
determinant is positive and \(K\) is strictly convex.
Hence the equation is uniformly elliptic on the compact sphere.

We improve the regularity of \(h\) by successive applications of
\cite[Lemma~17.16]{GilbargTrudinger}. Indeed, if
\(h\in C^{k+2,\beta}( S^{n-1})\) for some integer \(k\ge0\),
then the right-hand side of the preceding equation belongs to
\(C^{k+1,\beta}( S^{n-1})\). Using the established uniform ellipticity, the lemma applied in local coordinates gives
\[
    h\in C^{k+3,\beta}(S^{n-1}).
\]
Starting from \(h\in C^{2,\beta}( S^{n-1})\), induction
therefore yields \(h\in C^\infty( S^{n-1})\).

Finally, substituting the smooth formula for \(dS_{\gamma_\sigma,K}\) into \eqref{eq:weak-Euler} gives
\[
      {\gamma_\sigma^n(K^\circ)} e^{-(h^2+|\nabla_Sh|^2)/2\sigma^2}
        \det(\nabla_S^2h+hI)
        =
       {\gamma_\sigma^n(K)} e^{-1/(2\sigma^2h^2)}h^{-(n+1)}.
\]
This is the claimed Euler--Lagrange equation.
\end{proof}

\section{Analyzing Balls}

\subsection{Balls cannot be maximizers for $\sigma^2> {\frac{2}{n+1}}$}
\begin{proposition}
Let $n\ge 1$ and $\sigma^2>{\frac{2}{n+1}}$, then the Euclidean unit ball \(B_2^n\) is not a maximizer of $G^n_\sigma$. 
\end{proposition}

\begin{proof}
Fix \(e\in S^{n-1}\), and for \(|\eps|<1\) set
$K_\eps=B_2^n+\eps e$. 
Then \(0\in \operatorname{int}K_\eps\), and the support function of
\(K_\eps\) is $h_\eps(u)=1+\eps\langle u,e\rangle$. 
We expand the Gaussian measure of \(K_\eps\). Since
\[
        \gamma_\sigma^n(K_\eps)
        =
        c_{n,\sigma} \int_{B_2^n}
        e^{-|x+\eps e|^2/(2\sigma^2)}\,dx ,\quad \text{where} \ c_{n,\sigma}=(2\pi\sigma^2)^{-n/2},
\]
the first variation vanishes by symmetry. Moreover, by the divergence theorem, 
\[
\begin{aligned}
        \frac{d^2}{d\eps^2}
        \gamma_\sigma^n(B_2^n+\eps e)\bigg|_{\eps=0}
        &=
        c_{n,\sigma} \int_{B_2^n}\partial_{ee}
        \left(
        e^{-|x|^2/(2\sigma^2)}\right)\,dx  \\
        &=
        c_{n,\sigma}\int_{\partial B_2^n}
        \partial_e
        \left(
        e^{-|x|^2/(2\sigma^2)}\right)
        \langle u,e\rangle\,d\omega(u)  \\
        &=
        -c_{n,\sigma}\frac{e^{-1/(2\sigma^2)}}{\sigma^2}
        \int_{  S^{n-1}}\langle u,e\rangle^2\,d\omega(u).
\end{aligned}
\]
Hence
\[
        \gamma_\sigma^n(K_\eps)
        =
        \gamma_\sigma^n(B_2^n)
        -
        c_{n,\sigma}\frac{e^{-1/(2\sigma^2)}}{\sigma^2}
        \int_{  S^{n-1}}\langle u,e\rangle^2\,d\omega(u)\frac{\eps^2}{2}
        +
        o(\eps^2).
\]
For the dual, consider 
\[
        \rho_{K_\eps^\circ}(u)=\frac1{h_\eps(u)}
        =
        \frac1{1+\eps\langle u,e\rangle}.
\]
Denote
\[
        F(r)
        =
        c_{n,\sigma}
        \int_0^r e^{-s^2/(2\sigma^2)}s^{n-1}\,ds  
\]
and check that 
\[
        F'(1)=c_{n,\sigma}e^{-1/(2\sigma^2)},
        \qquad
        F''(1)=c_{n,\sigma}e^{-1/(2\sigma^2)}\left(n-1-\frac1{\sigma^2}\right). 
\]

We write 
\[
        \gamma_\sigma^n(K_\eps^\circ)
        =
        \int_{ S^{n-1}}
        F\left(\frac1{1+\eps\langle u,e\rangle}\right)
        d\omega(u),\]
and since  $\frac1{1+\eps\langle u,e\rangle}
        =
        1-\eps\langle u,e\rangle
        +\eps^2\langle u,e\rangle^2
        +
        o(\eps^2)$, we get, using
\(\int_{S^{n-1}}\langle u,e\rangle\,d\omega(u)=0\), that
\[
\begin{aligned}
        \gamma_\sigma^n(K_\eps^\circ)
        &=
        \gamma_\sigma^n(B_2^n)
        +
        c_{n,\sigma}e^{-1/(2\sigma^2)} \left(\int_{  S^{n-1}}\langle u,e\rangle^2 d\omega(u)\right)\eps^2
        \\&+
      c_{n,\sigma}\frac{e^{-1/(2\sigma^2)}}{2} 
        \left(n-1-\frac1{\sigma^2}\right)\left(\int_{  S^{n-1}}\langle u,e\rangle^2\,d\omega(u)\right)\eps^2
        +
        o(\eps^2)  \\
        &=
        \gamma_\sigma^n(B_2^n)
        +
        c_{n,\sigma}\frac{e^{-1/(2\sigma^2)}}{2}
        \left(n+1-\frac1{\sigma^2}\right)\left(\int_{ S^{n-1}}\langle u,e\rangle^2\,d\omega(u)\right)\eps^2
        +
        o(\eps^2).
\end{aligned}
\]

Multiplying the two expansions gives
\[
\begin{aligned}
        \gamma_\sigma^n(K_\eps)
        \gamma_\sigma^n(K_\eps^\circ)
        &=
        (\gamma_\sigma^n(B_2^n))^2
        +
        \gamma_\sigma^n(B_2^n) c_{n,\sigma}\frac{ e^{-1/(2\sigma^2)}}{2}
        \left(n+1-\frac2{\sigma^2}\right)
        \left(\int_{  S^{n-1}}\langle u,e\rangle^2 d\omega(u)\right)\eps^2
        +
        o(\eps^2).
\end{aligned}
\]
We see that when $\sigma^2>\frac{2}{n+1}$, for all sufficiently small nonzero \(\eps\) we get
\[
        \gamma_\sigma^n(K_\eps)\gamma_\sigma^n(K_\eps^\circ)
        >
        \gamma_\sigma^n(B_2^n)^2 .
\]
Thus \(B_2^n\) is not a maximizer.
\end{proof}

\subsection{In the plane: disks are maximizers up to $\sigma^2 = \frac{2}{3}$}

In this section and the next we use the Euler--Lagrange equation derived in Proposition~\ref{prop:regularity-maximizer}: for a maximizer \(K\) with support function \(h\) 
\begin{equation}\label{EL}
 \det(\nabla_S^2 h+hI)=
        C h^{-(n+1)}
        \exp\left[
        \frac{1}{2\sigma^2}
        \left(h^2+|\nabla_Sh|^2-h^{-2}\right)
        \right],
\end{equation}
where
\[
        C=\frac{\gamma_\sigma^n(K)}{\gamma_\sigma^n(K^\circ)}>0.
\]

\begin{proposition}
\label{prop:planar}
Let \(n=2\). If
$ \sigma^2\le \frac23,
      $
then the unique maximizer of $G^n_\sigma$
is the Euclidean unit disk \(B_2^2\).
\end{proposition}

\begin{proof}
Let \(K\) be such a maximizer, in particular it is smooth and strictly convex and has a reflection symmetry. Let \(h\) be its support function
on \(  S^1\). We write derivatives with respect to the angular
variable \(\theta\), and set
\[
        \rho=h+h''.
\]
Since \(K\) is smooth and strictly convex, \(\rho>0\). In dimension two, \eqref{EL} becomes
\[
        \rho
        =
        C h^{-3}
        \exp\left[
        \frac{1}{2\sigma^2}
        \left(h^2+(h')^2-h^{-2}\right)
        \right].
\]

We show that \(h\) must be constant. Put
\[
        u=h'.
\]
Differentiating the logarithm  and using that $ \rho'=h'+h'''=u+u''$
we obtain
\[
        \frac{\rho'}{\rho}
        =
        u\left[
        \frac{1}{\sigma^2}(\rho+h^{-3})-\frac3h
        \right].
\]
Thus
\[
        u''
        +
        \left(
        1-\rho\left[
        \frac{1}{\sigma^2}(\rho+h^{-3})-\frac3h
        \right]
        \right)u
        =
        0.
\]

Assume towards a contradiction that \(h\) is not constant, namely  \(u\not\equiv0\).
Let \(I=(a,b)\) be a connected component of \(\{u\ne0\}\) so that 
$ u(a)=u(b)=0$ and $u$ is of constant sign on $I$. 
Multiplying the last identity by \(u\) and integrating
over \(I\), we get
\[
        \int_I (u')^2\,d\theta
        =
        \int_I u^2\,d\theta
        -
        \int_I
        \rho\left[
        \frac{1}{\sigma^2}(\rho+h^{-3})-\frac3h
        \right]u^2\,d\theta .
\]

We claim that the second term on the right, which we denote \(J_I\), is positive whenever \(\sigma^2\le 2/3\). Indeed, since
$\rho=h+u'$,
we have
\[
\begin{aligned}
J_I
={}&
\int_I
(h+u')
\left[
\frac{1}{\sigma^2}(h+u'+h^{-3})-\frac3h
\right]u^2\,d\theta .
\end{aligned}
\]
Expanding and integrating by parts the terms containing \(u'\), using
\(u(a)=u(b)=0\), gives
\[
\begin{aligned}
J_I
={}&
\frac{1}{\sigma^2}\int_I
u^2\left(u'+h-h^{-1}\right)^2\,d\theta  
+
(\frac{2}{\sigma^2}-3)\int_Iu^2\,d\theta  +
\int_I
\left[
\frac{h^{-4}}{\sigma^2}  
+
\frac{2/\sigma^2-3}{3}h^{-2}
\right]u^4\,d\theta .
\end{aligned}
\]

If \(\sigma^2\le 2/3\), all terms in the expression for $J_I$ are non-negative, and the last
term is strictly positive because \(u\not\equiv0\) on \(I\), showing that  $J_I>0$. 

Thus we have shown 
\[
        \int_I (u')^2\,d\theta
        <
        \int_I u^2\,d\theta.
\]
On the other hand, the  Poincar\'{e} inequality on  
\(I=(a,b)\) gives
\[
        \int_I (u')^2\,d\theta
        \ge
        \frac{\pi^2}{(b-a)^2}
        \int_I u^2\,d\theta .
\]
Combining the two we get that 
$  b-a>\pi$. 

Thus every connected component of \(\{u\ne0\}\) has length strictly larger
than \(\pi\). This is impossible since \(u=h'\) is periodic and
has vanishing integral on \( S^1\).
Therefore \(u\equiv0\), and hence \(h\) is constant. Thus \(K=rB_2^2\)
for some \(r>0\). Among balls $rB_2^2$, the Gaussian product
is uniquely maximized at \(r=1\) (this follows from Theorem \ref{thm:symmetric} but is also a simple consequence of log concavity). Hence \(K=B_2^2\).
\end{proof}

\subsection{In dimension $n\ge3$: balls must be maximizers for $\sigma^2\le 1/n$}

\begin{proposition}
\label{prop:general-rigidity-monotone-range}
Let \(n\ge2\). If $
        \sigma^2\le \frac1n$
then the unique maximizer of $G^n_\sigma$ 
 is the
Euclidean unit ball \(B_2^n\).
\end{proposition}

\begin{proof}
Let \(K\) be such a maximizer, and let \(h=h_K\) be its support function.
Set
\[
        M=\max_{S^{n-1}}h,
        \qquad
        m=\min_{S^{n-1}}h.
\]
At a maximum point \(u_M\) of \(h\), we have
\[
        \nabla_Sh(u_M)=0,
        \qquad
        \nabla_S^2h(u_M)\le0.
\]
Since \(h(u_M)=M\), this gives
\[
        0<\nabla_S^2h(u_M)+MI\le MI.
\]
Therefore, $\det(\nabla_S^2h(u_M)+MI)\le M^{n-1}$.
Evaluating \eqref{EL} at \(u_M\), we get
\[
        C M^{-(n+1)}
        \exp\left[
        \frac{1}{2\sigma^2}(M^2-M^{-2})
        \right]
        \le
        M^{n-1}.
\]
Equivalently,
\[
        C
        \le
        M^{2n}
        \exp\left[
        -\frac{1}{2\sigma^2}(M^2-M^{-2})
        \right].
\]

Similarly, at a minimum point \(u_m\) of \(h\),
$\nabla_Sh(u_m)=0$, and
$ \nabla_S^2h(u_m)\ge0$. Thus
\[
        \nabla_S^2h(u_m)+mI\ge mI,
\]
and so $\det(\nabla_S^2h(u_m)+mI)\ge m^{n-1}$.
Evaluating \eqref{EL} at \(u_m\), we obtain
\[
        C m^{-(n+1)}
        \exp\left[
        \frac{1}{2\sigma^2}(m^2-m^{-2})
        \right]
        \ge
        m^{n-1}.
\]
Hence, $C \ge m^{2n} \exp\left[ -\frac{1}{2\sigma^2}(m^2-m^{-2})\right]$.

Define
\[
        F_\sigma(r)
        =
        r^{2n}
        \exp\left[
        -\frac{1}{2\sigma^2}(r^2-r^{-2})
        \right],
        \qquad r>0.
\]
The upper and lower bounds for $C$ are thus 
\[
        F_\sigma(m)\le C\le F_\sigma(M).
\]
Using the assumption \(\sigma^2\le 1/n\), we see
\[
        \frac{d}{dr}\log F_\sigma(r)
        =
        \frac{2n}{r}
        -
    \frac{1}{\sigma^2}(r+r^{-3})\le \frac{2n}{r}
        -
        \frac{2}{\sigma^2r}
        \le0.
\]

Hence \(F_\sigma\) is non-increasing on \((0,\infty)\), and in fact is decreasing. Since \(m\le M\), this gives
\[
        F_\sigma(m)\ge F_\sigma(M).
\]
However, the upper and lower bounds for $C$ implied the opposite inequality and this means we must have equality throughout, that is, 
$
        F_\sigma(m)=F_\sigma(M)$. But 
\(F_\sigma\) is decreasing, so this means $
        m=M$.
Therefore \(h\) is constant, and \(K=rB_2^n\) for some \(r>0\). As in the two dimensional case, for Euclidean balls $G^n_\sigma$
is uniquely maximized at \(r=1\), resulting in the desired consequence,  \(K=B_2^n\).
\end{proof}

\begin{remark}
The gap between \(\frac1n\) and \(\frac2{n+1}\) is precisely the range where the
one-dimensional monotonicity of \(F_\sigma\) fails, and where for the moment we are unable to show that the Euclidean unit ball is a maximizer. 
\end{remark}

\bibliographystyle{plain}

\bibliography{ref}

\end{document}